\documentclass[aip,cha,amsmath,amssymb,reprint,groupedaddress]{revtex4-1}
\usepackage{graphicx,subfigure}
\usepackage{amsfonts,amsbsy,amscd,amsgen,amsthm,dsfont}
\usepackage[]{natbib}
\usepackage[colorlinks=true,allcolors=blue]{hyperref}
\usepackage{bm}
\theoremstyle{definition}
\usepackage{tabularx}
\usepackage{enumitem}
\usepackage[utf8]{inputenc}
\usepackage[T1]{fontenc}

\newcommand{\id}{\mathrm d}
\newcommand{\vc}{\bm}
\newcommand{\bnabla}{\pmb\nabla}

\renewcommand{\tilde}{\widetilde}

\renewcommand{\Pr}{\mathbb P}
\newcommand{\bsigma}{\vc \sigma}

\DeclareMathAlphabet\mathbfcal{OMS}{cmsy}{b}{n}

\newtheorem{thm}{Theorem}
\newtheorem{rem}{Remark}

\newtheorem{defn}{Definition}

\newtheorem{cor}{Corollary}

\begin{document}
\title{Mitigation of tipping point transitions by time-delay feedback control}
\author{Mohammad Farazmand}
\email{farazmand@ncsu.edu}
\affiliation{Department of Mathematics, North Carolina State University,
	2311 Stinson Drive, Raleigh, NC 27695-8205, USA}

\begin{abstract}
In stochastic multistable systems driven by the gradient of a potential, transitions between equilibria 
is possible because of noise. We study the ability of linear delay feedback control to 
mitigate these transitions, ensuring that the system stays near a desirable equilibrium. For small delays, we show that the 
control term has two effects: i) a stabilizing effect by deepening the potential well around the desirable equilibrium, and ii) a destabilizing effect by intensifying the noise by a factor of $(1-\tau\alpha)^{-1/2}$, 
where $\tau$ and $\alpha$ denote the delay and the control gain, respectively.
As a result, successful mitigation depends on the competition between these two factors. 
We also derive analytical results that elucidate the choice of the appropriate control gain and delay that ensure successful mitigations. These results eliminate the need for any Monte Carlo simulations 
of the stochastic differential equations, and therefore significantly reduce the computational
cost of determining the suitable control parameters. We demonstrate the application of our results on two examples. 
\end{abstract}

\maketitle

\begin{quotation}
Many complex systems, such as the climate, stock markets, 
population of species, and the nervous system, have multiple stable equilibrium states. 
Near the tipping point of the system, 
small amounts of stochastic disturbances can lead to transitions between 
these equilibria. Typically, one of the equilibria 
is desirable and transitions away from it lead to catastrophic consequences. 
Here, we analyze the ability of a delay feedback control 
to mitigate transitions away from a desirable equilibrium.  
\end{quotation}

\section{Introduction}
Tipping point transitions refer to abrupt changes in the state of 
a dynamical system caused by a rapid transition from an equilibrium state
of the system to another equilibrium. 
Such transitions have been observed in many systems
such as climate dynamics~\cite{lenton2011,dijkstra2019}, 
epilepsy~\cite{McSharry2003, meisel2015, Wilkat2019}, 
population ecology~\cite{drake2010,Dai2012,Botero2015},
fluid dynamics~\cite{Farazmande1701533, blonigan2019},
and stock markets~\cite{Longin1996,sornette2012,jurczyk2017}. 
Typically, one of the system equilibria 
is desirable and transitions away from it are catastrophic. 
As such, it is necessary to devise control methods that mitigate 
such transitions. 

Tipping point transitions are usually studied in the context of 
bifurcation theory, where the system parameters
change gradually. Under certain parameter variations, a bifurcation 
takes place whereby an equilibrium loses stability and the system 
evolves towards another stable equilibrium.

However, even before such bifurcations take place, the system can 
switch equilibria due to small amounts of stochastic disturbances. 
Such stochasticity-induced transitions are common in multistable 
systems, where several stable equilibria exist simultaneously. 
These transitions become more feasible near the tipping point of the system (i.e., close to the bifurcation).

Here, we focus on transitions in stochastic multistable systems.
These systems are generally modeled by the stochastic differential equation (SDE), 
\begin{equation}
\id\vc X = \vc f(\vc X)\id t + \bsigma(\vc X)\id \vc W,
\label{eq:ode_general}
\end{equation}
where $\vc X(t)$ denotes the state of the system at time $t$, 
$\vc f:\mathbb R^n\to\mathbb R^n$ is the driving force, $\vc W(t)$ is the 
Brownian motion (or noise), and $\bsigma(\vc X)$ is the diffusion tensor determining
the intensity of the noise. The system is multistable if, in the absence of noise, there are several 
stable fixed points. 

Much attention has been paid to real-time prediction~\cite{scheffer2008, scheffer2009, scheffer2010,Scheffer2012,farazmand2019a} 
and quantification~\cite{pratt1986, vanden2006,Zhou2012} of transitions in system~\eqref{eq:ode_general}. 
Real-time prediction refers to early warning signs of the upcoming transitions that enable their predictions. 
Quantification, on the other hand, refers to estimating the probability of transitions in the aggregate,
and the most likely paths such transitions may take. 
However, mitigation of tipping point transitions has received relatively little attention. 
\begin{figure}
	\centering
	\includegraphics[width=.3\textwidth]{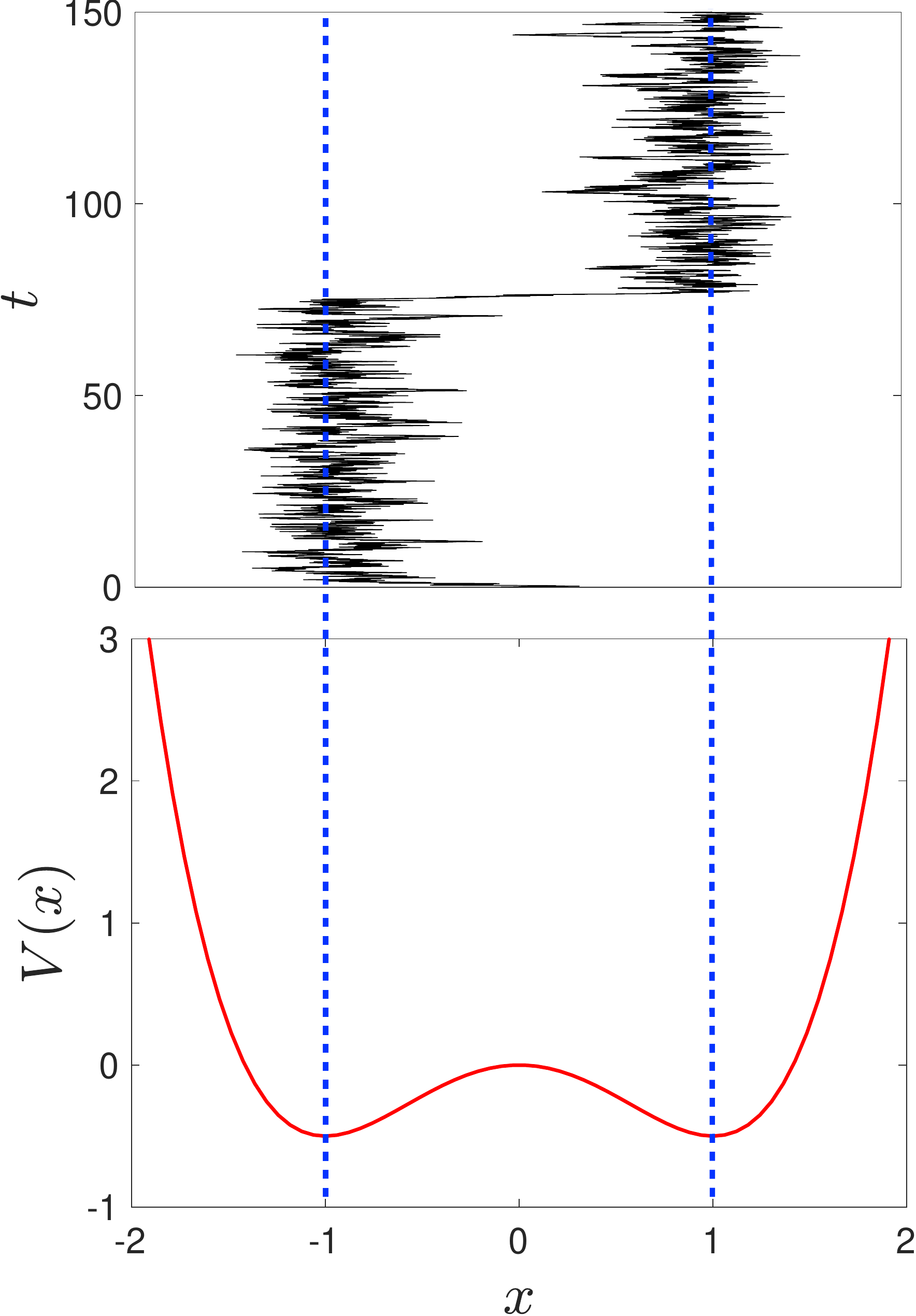}
	\caption{Example of abrupt transitions in a stochastic multistable system.
		Shown are a one-dimensional potential $V$ (bottom panel) and a trajectory $X(t)$ (top panel) of the 
		corresponding SDE~\eqref{eq:ode_general} with $\vc f = -\bnabla V$. The blue dashed lines
		mark the equilibrium states of the system. A detailed description of this system is given in Section~\ref{sec:DW_1d}.}
	\label{fig:trans_intro}
\end{figure}

Here, we analyze the ability of linear time-delay feedback control in mitigating 
transitions in stochastic multistable system. We consider the special case of 
a system driven by the gradient of a potential $V$, such that $\vc f(\vc x) = -\bnabla V(\vc x)$. 
In the absence of noise ($\bsigma=\vc 0$), all minima of the potential $V$ are stable fixed points.
However, small amounts of noise cause rare transitions between these stable equilibria
(see figure~\ref{fig:trans_intro}, for an example). 

Delay feedback control has been used widely in deterministic systems
to stabilize otherwise unstable fixed points or periodic orbits~\cite{ott1990, pyragas1992,Pyragas1995,Ahlborn2004}.
Recently, Suresh and Chandrasekar~\cite{Suresh2018} used delay feedback to mitigate extreme events 
in a deterministic system driven by a time-periodic force. However, little is known about the ability of 
delay feedback control to also mitigate critical transitions in stochastic multistable systems. 

Here, we discover several properties of such a control strategy. The controlled system is 
a stochastic delay differential equation (SDDE). As such, its mathematical analysis is 
extremely difficult due to the infinite-dimensional nature of the problem. For small delay, 
we derive an approximating SDE whose behavior mimics that of the controlled system. 
This approximation reveals several properties of the control. In particular, we show that 
the delay feedback has two simultaneous effects. First, the control term
modifies the potential $V$ into an effective potential $\tilde V$ which is deeper around the 
desirable equilibrium. As a result, this equilibrium becomes more stable, and 
transitions away from it become less likely. 

At the same time, the delay effectively enhances the noise intensity by 
a factor of $1/\sqrt{1-\alpha\tau}$, where $\alpha$ is the control gain and $\tau$ is the delay. 
This intensified noise increases the probability of transitions away from the desirable equilibrium.
The success of the delay feedback control depends on the interplay between these two
competing factors. Therefore, the appropriate choice of the gain and the delay becomes crucial.

We derive analytical results that inform this choice. In particular, we derive the exact stationary density 
of the approximating SDE, and an upper bound for the probability of transitions away from the desirable equilibrium. 
These two analytical results elucidate the proper choice of the control 
parameters at a low computational cost, and eliminate
any need for Monte Carlo simulations of the underlying SDDE (or its approximating SDE). 

This paper is organized as follows. 
In Section~\ref{sec:prelim}, we introduce the set-up of the problem. 
In Section~\ref{sec:approxSDE}, we derive the approximating SDE of the controlled system.
Stationary distribution of the approximating SDE is presented in Section~\ref{sec:p0}
and an upper bound for the probability of transitions is derived in Section~\ref{sec:prob_trans}. 
In Section~\ref{sec:numerics}, we demonstrate our results with two numerical examples. 
Finally, Section~\ref{sec:concl} contains our concluding remarks and future directions. 

\section{Preliminaries and set-up}\label{sec:prelim}
In this section, we introduce the notation and set-up of the problem. 
In particular, we describe the uncontrolled system and its controlled 
counterpart, and outline their general properties. 
\subsection{Uncontrolled system}
Consider a stochastic process $\vc X(t)$ that is generated by the SDE,
\begin{equation}
\id\vc X = -\bnabla V(\vc X)\id t + \bsigma(\vc X) \id{\vc W},\quad \vc X(0)=\vc x^0,
\label{eq:ode}
\end{equation}
where $V:\mathbb R^n\to \mathbb R$ 
is the potential function and $\vc x^0$ denotes the initial condition of the stochastic process. 
Components of the vector $\vc W(t)\in\mathbb R^m$ ($ m\geq n$) are mutually independent,
standard Wiener processes so that the increments $W_i(t_2)-W_i(t_1)$ are Gaussian,
satisfying
\begin{subequations}
\begin{equation}
\mathbb E[W_i(t_2)- W_i(t_1)] = 0,
\end{equation}
\begin{equation}
\mathbb E[(W_i(t_2)- W_i(t_1))(W_j(t_2)- W_j(t_1))] = (t_2-t_1)\delta_{ij},
\end{equation}
\end{subequations}
for all $i,j\in\{1,2,\cdots,m\}$ and $0\leq t_1\leq t_2$. Here, $\delta_{ij}$ denotes the Kronecker delta.
The diffusion tensor $\bsigma(\vc x)\in\mathbb R^{n\times m}$ is a function of $\vc x$.
We assume that the potential $V$ and the tensor $\bsigma$ are smooth enough to ensure the
regularity of the SDE~\eqref{eq:ode}.
  
Consider the case where the potential $V$ has a finite number of disjoint local minima 
$\vc x_1, \vc x_2,\cdots, \vc x_K \in\mathbb R^n$.
In the absence of noise ($\bsigma\equiv \vc 0$), the local minima $\vc x_k$ ($k=1,2,\cdots,K$) are all locally stable fixed points 
of the ODE~\eqref{eq:ode}. The presence of noise $(\bsigma\neq \vc 0)$ allows for rare transitions
between these fixed points. We assume that one of these equilibria is desirable 
and would like to suppress the transition of system trajectories to other equilibria.
We denote the desirable fixed point by $\vc x_a$, and note that $\vc x_a=\vc x_k$ for some $k\in\{1,2,\cdots,K\}$.

Our goal is to mitigate the transitions away from the desirable fixed point $\vc x_a$, 
or at least reduced the probability of such transitions. 

\subsection{Controlled system}
In order to suppress transitions away from the desirable equilibrium $\vc x_a$, 
we add a time-delay feedback control to system~\eqref{eq:ode}.
The controlled system reads
\begin{equation}
\id {\mathbfcal{X}} = -\bnabla V(\mathbfcal X)\id t - \vc u(t-\tau)\id t +\bsigma(\mathbfcal X) \id{\vc W},
\label{eq:ode_cont}
\end{equation}
where $\tau>0$ is the delay and $\vc u(t-\tau)$ is the control term.
We denote the trajectories of the controlled system~\eqref{eq:ode_cont}
by $\mathbfcal X(t)$ to distinguish them from the trajectories $\vc X(t)$ of the uncontrolled system~\eqref{eq:ode}. The delay $\tau$ is desirable since in practice the time required 
for state measurements, and consequent actuation of the control, 
always imposes a time lag~\cite{dorf2011,farazmand2019b}. 

Here we consider linear control terms of the form
\begin{equation}
\vc u(t-\tau) = \alpha\left(\mathbfcal X(t-\tau)-\vc x_a\right),
\label{eq:u}
\end{equation}
where $\alpha>0$ is the control gain and $\vc x_a$ is the desirable equilibrium.
Note that with this control term, Eq.~\eqref{eq:ode_cont} is an SDDE. This SDDE must be supplied with the appropriate initial condition, 
$$\mathbfcal X(t) = \vc x^0(t),\quad t\in [-\tau,0],$$
where the function $\vc x^0:[-\tau,0]\to\mathbb R^n$ denotes the initial data. 

The delay feedback control~\eqref{eq:u} has been used extensively for stabilizing fixed points and 
periodic orbits in deterministic systems~\cite{ott1990,pyragas1992}. Here, we investigate their effectiveness 
in suppressing transitions away from the desirable state $\vc x_a$ in the 
stochastic context.

The delay term introduces significant difficulties in analysis of the controlled system.
In particular, the phase space of an SDDE is infinite-dimensional, and furthermore,
the Fokker--Planck equation for the density is an integro-differential equation~\cite{guillouzic1999,Frank2005}.
In order to circumvent these difficulties we derive an SDE which approximates 
the SDDE~\eqref{eq:ode_cont} in the small delay limit $0<\tau\ll 1$.
The approximating SDE is more amenable to mathematical analysis.

\section{Approximating SDE}\label{sec:approxSDE}
To derive an approximate SDE for the SDDE~\eqref{eq:ode_cont}, we expand the
control term $\mathbfcal X(t-\tau)$ in $\tau$. This type of Taylor expansion 
for the delay term has been used routinely for delay differential equations~\cite{Mazanov1974,Suresh2018}, although 
its mathematical justification remains an open problem.
Mazanov and Tognetti~\cite{Mazanov1974} point out the potential pitfalls of this approximation, namely that if too many
terms in the Taylor series are kept, the resulting approximating equation may become unstable.
However, Insperger~\cite{Insperger2015} argues with numerical examples that if $\mathcal O(\tau^2)$ 
terms are negligible, and one only keeps the $\mathcal O(\tau)$ terms, the resulting SDE is a reasonable
approximation of the original SDDE, and preserves many of its stability properties. 

Here, we apply this approximation. 
For short delays $0<\tau\ll 1$, we consider the Taylor expansion
\begin{equation}
\mathbfcal X(t-\tau) =\mathbfcal X(t)-\tau \dot{\mathbfcal  X}(t) +\mathcal O(\tau^2).
\label{eq:taylor}
\end{equation}
Substituting this expression in Eq.~\eqref{eq:ode_cont} and neglecting $\mathcal O(\tau^2)$ terms, we obtain
\begin{equation}
(1-\alpha\tau )\id{\mathbfcal X} = -\bnabla V(\mathbfcal X) \id t- \alpha\left(\mathbfcal X-\vc x_a\right)\id t +\bsigma(\mathbfcal X)\id{\vc W}. 
\label{eq:approx_sde}
\end{equation}
We refer to Eq.~\eqref{eq:approx_sde} as the \emph{approximating SDE} since
it approximates the SDDE~\eqref{eq:ode_cont}.
We further assume that $\alpha\tau<1$, since otherwise 
the desirable state $\vc x_a$ becomes unstable.

The following theorem shows that the approximating SDE can be transformed into 
an equivalent SDE by rescaling time. The rescaled equation has the form of a
standard SDE.

\begin{thm}\label{thm:RASDE}
The trajectories of the approximating SDE~\eqref{eq:approx_sde} coincide
with the trajectories of the \emph{rescaled approximating SDE}, 
\begin{equation}
\id\mathbfcal X = -\bnabla V(\mathbfcal X)\id s - \alpha \left(\mathbfcal X-\vc x_a\right)\id s +\tilde\bsigma(\mathbfcal X)\id\vc W(s),
\label{eq:rescaled_approx_sde}
\end{equation}
for the process $\mathbfcal X(s)$, where $s = t/(1-\alpha\tau)$ is the rescaled time,
$\vc W(s)$ is a standard Wiener process, and 
$\tilde \bsigma(\vc x) := \bsigma(\vc x) /\sqrt{1-\alpha\tau}$.
\end{thm}
\begin{proof}
See Appendix~\ref{app:RASDE}.
\end{proof}

\begin{table*}[t]
\centering
\caption{List of equations. The controlled system is a stochastic delay differential equation, whereas the rest are stochastic differential equations without delay.}
\renewcommand{\arraystretch}{2}
\begin{tabularx}{.7\textwidth}{p{5.5cm} | c | c}
Equation & Terminology &	Parameters\\ 
\hline\hline
$\begin{array}{lcl}
\id\vc X &=& -\bnabla V(\vc X)\id t + \bsigma(\vc X) \id{\vc W}
\end{array}$
& Uncontrolled system (SDE) & ---\\
\hline
$\begin{array}{lcl}
\id {\mathbfcal{X}} &=& -\bnabla V(\mathbfcal X)\id t\\ 
& &- \alpha\left(\mathbfcal X(t-\tau)-\vc x_a\right)\id t\\ 
& & +\bsigma(\mathbfcal X) \id{\vc W}
\end{array}$
& Controlled system (SDDE) & $\alpha, \tau>0$\\
\hline
$\begin{array}{lcl}
(1-\alpha\tau )\id{\mathbfcal X} &= & -\bnabla V(\mathbfcal X) \id t\\
& &- \alpha\left(\mathbfcal X-\vc x_a\right)\id t\\
& & +\bsigma(\mathbfcal X)\id{\vc W}
\end{array}$
& Approximating SDE & $0<\alpha\tau<1$
\\
\hline
$\begin{array}{lcl}
\id{\mathbfcal X} &= & -\bnabla V(\mathbfcal X)\id s\\
& & - \alpha \left(\mathbfcal X-\vc x_a\right)\id s\\
& & +\tilde\bsigma(\mathbfcal X)\id\vc W
\end{array}$
& Rescaled approximating SDE & 
$\begin{array}{l}
s = t/(1-\alpha\tau), \\
\tilde\bsigma = \bsigma/\sqrt{1-\alpha\tau}
\end{array}$
\\
\hline
\end{tabularx}
\renewcommand{\arraystretch}{1}
\label{tab:eqs}
\end{table*}

Defining the \emph{effective potential} $\tilde V$,  
\begin{equation}
\tilde V(\vc x) := V(\vc x) +\frac{\alpha}{2}|\vc x-\vc x_a|^2,
\label{eq:eff_potential}
\end{equation}
the rescaled approximating SDE~\eqref{eq:rescaled_approx_sde} can be written more compactly as
\begin{equation}
\id\mathbfcal X = -\bnabla \tilde V(\mathbfcal X)\id s +\tilde\bsigma(\mathbfcal X)\id\vc W.
\end{equation}
Table~\ref{tab:eqs} contains a list of relevant equations, their names, and their parameter range.

\begin{rem}
Two remarks about Theorem~\ref{thm:RASDE} are in order:
\begin{enumerate}[label=(\alph*)]
\item The control term modifies the potential $V$ in the original SDE~\eqref{eq:ode}
to the effective potential~\eqref{eq:eff_potential}. The term $\alpha |\vc x-\vc x_a|^2/2$
increases the potential gap around the desirable state $\vc x_a$, 
deepening the potential well. As a result, 
the control further stabilizes the state $\vc x_a$, and consequently
transitions away from this state become harder. 

\item At the same time, the control intensifies the noise
by a factor of $1/\sqrt{1-\alpha\tau}$. In principle, this can facilitate 
transitions away from the desirable state $\vc x_a$.
\end{enumerate}
\end{rem}
Therefore, the delay control introduces two competing factors: One 
modifies the potential to further stabilize the desirable state $\vc x_a$, 
whereas the other increases the noise intensity. 
In the next two sections, we quantify the contribution from each factor and 
identify the regime where the stabilizing effect outperforms the 
effect of intensified noise. 

We also note that the delay $\tau$ could in principle vanish, $\tau=0$. 
In this case, the control gain $\alpha$ can be arbitrarily large and the difficulty of 
approximating a delay differential equation will be avoided. 
However, from a practical point of view, the existence of delay is 
appealing since the time required for the
state measurement, and the subsequent actuation of the control, inevitably 
introduces a delay.

\section{Stationary distribution}\label{sec:p0}
In this section, we review some well-known results about the probability density
associated with the SDE~\eqref{eq:rescaled_approx_sde}. These results are subsequently
used in Section~\ref{sec:prob_trans} to derive estimates of the transition probability away from the desirable state $\vc x_a$.

We denote the probability density function (PDF) associated with the 
rescaled approximating SDE~\eqref{eq:rescaled_approx_sde}
by $p(\vc x,t)$ such that $\mathbb P(\mathbfcal X(t)\in\mathcal S) = \int_{\mathcal S}p(\vc x,t)\id \vc x$,
for any Lebesgue-measurable set $\mathcal S\subset \mathbb R^n$.
This probability density, satisfies the Fokker--Planck equation
(or forward Kolmogorov equation)~\cite{mackey1994chaos},
\begin{equation}\label{eq:fp_general}
\partial_t p=\bnabla\cdot \left[\bnabla\tilde V(\vc x)p \right]+
\frac12\bnabla\cdot\left[\bnabla\cdot\left(\tilde\bsigma(\vc x)\tilde\bsigma(\vc x)^\top p \right)\right], 
\end{equation}
accompanied with some initial condition $p(\cdot,0)=f$
where $f:\mathbb R^n\to\mathbb R^+$ 
satisfies $\int_{\mathbb R^n}f(\vc x)\id \vc x=1$.
It follows from the definition of the probability density $p$ that it is positive semi-definite for all times
and its integral over the entire state space is $1$, i.e., 
\begin{subequations}\label{eq:p_prop}
\begin{equation}
p(\vc x,t)\geq 0, \quad \forall \vc x\in\mathbb R^n, \quad \forall t\geq 0,
\end{equation}
\begin{equation}
\int_{\mathbb R^n}p(\vc x, t)\id\vc x=1,\quad \forall t\geq 0.
\end{equation}
\end{subequations}
Note that for $\alpha=0$, the Fokker--Plank equation~\eqref{eq:fp_general} 
determines the evolution of the density $p$ 
corresponding to the uncontrolled system~\eqref{eq:ode}.

We are particularly interested in stationary densities $p_0$ 
which are invariant in time and describe the asymptotic evolution of the 
stochastic process. 

\begin{defn}
A function $p_0:\mathbb R^n\to\mathbb R^+$ is a
stationary density of the rescaled approximating SDE~\eqref{eq:rescaled_approx_sde},
if $p_0$ satisfies
\begin{equation}
\bnabla\cdot \left[\bnabla\tilde V(\vc x)p_0 \right]+
\frac12\bnabla\cdot\left[\bnabla\cdot\left(\tilde\bsigma(\vc x)\tilde\bsigma(\vc x)^\top p_0 \right)\right]=0,
\label{eq:fp_eq}
\end{equation}
and $\int_{\mathbb R^n}p_0(\vc x)\id\vc x=1$. In other words, 
$p_0$ satisfies the Fokker--Planck equation~\eqref{eq:fp_general} and
$\partial_t p_0 =0$.
\end{defn} 

In the most general case of stochastic processes, the stationary distribution may not exist. 
Even when it does, it may not be unique. In our case, under reasonable assumptions, 
the existence and uniqueness of the equilibrium density is guaranteed. In addition, 
the stationary density has a known closed-form solution. For that, we need the following assumptions.

\begin{enumerate}[label=\textbf{(A\arabic*)}]
	\item We assume that the diffusion tensor has the form $\bsigma(\vc x)\equiv \sigma\vc \Sigma$
	where $\sigma>0$ and $\vc\Sigma\in \mathbb R^{n\times m}$ is a semi-orthogonal matrix such that 
	$\vc\Sigma\vc\Sigma^\top =\vc I_n$
	where $\vc I_n$ is the $n\times n$ identity matrix.
	\item 	We assume that $V\in C^2(\mathbb R^n;[a,\infty))$, for some finite $a\in\mathbb R$,
	and furthermore
	\begin{equation}
	\int_{\mathbb R^n} e^{-V(\vc x)}\id \vc x <\infty.
	\end{equation}
\end{enumerate}

Under assumption (A1), the Fokker--Planck equation~\eqref{eq:fp_eq} simplifies to
\begin{equation}\label{eq:fp_simple}
\partial_t p=\bnabla\cdot \left[\bnabla\tilde V(\vc x)p \right]+
\tilde \nu\Delta p, \quad
p(\cdot,0)=f,
\end{equation}
where $\Delta$ denotes the Laplacian, and the diffusion constant is given by
$\tilde \nu := \tilde\sigma^2/2$ where $\tilde\sigma := \sigma/\sqrt{1-\alpha\tau}$.

Assumption (A2) implies that the potential $V$ is bounded from below, $V(\vc x)\geq a$
for all $\vc x\in\mathbb R^n$. 
We may assume, without loss of generality, that $a\geq 0$. This is permissible since the graph of $V$ can be shifted upward, by an arbitrary positive constant, without altering the dynamics of~\eqref{eq:ode}.

Assuming (A1) and (A2), the following result holds.
\begin{thm}\label{thm:EU}
Assume (A1) and (A2). Then
\begin{enumerate}
\item 
The stochastic process $\mathbfcal X(s)$,
generated by the rescaled approximating SDE~\eqref{eq:rescaled_approx_sde}, 
has a unique stationary density.

\item The stationary density is given by 
\begin{equation}
p_0(\vc x) = C \exp\left[{-\frac{\tilde V(\vc x)}{\tilde \nu}}\right],
\label{eq:fp_sol}
\end{equation}
where $C$ is a normalizing factor defined by
\begin{equation}
C = \left[ \int_{\mathbb R^n}\exp\left({-\frac{\tilde V(\vc x)}{\tilde \nu}}\right)\id\vc x \right]^{-1},
\end{equation}
so that $\int_{\mathbb R^n} p_0(\vc x)\id \vc x =1$.
\item The stationary density $p_0$ is a globally asymptotically attracting solution 
of the Fokker--Planck equation~\eqref{eq:fp_general}.
\end{enumerate}
\end{thm}
\begin{proof}
This is a well-known result. The proof can be found in, e.g., Refs.~\onlinecite{Ambrosio2008,matkowsky1981,Risken1989}
\end{proof}

\begin{rem}
A few remarks on Theorem~\ref{thm:EU} are in order.
\begin{enumerate}[label=(\alph*)]
\item Theorem~\ref{thm:EU} holds for the uncontrolled system by setting $\alpha=0$, in which case 
$\tilde V=V$, and $\tilde\nu = \nu$ where $\nu := \sigma^2/2$.
\item The fact that the stationary density~\eqref{eq:fp_sol} is globally attracting implies that
all initial densities $f(\vc x)$ converge asymptotically to the unique stationary density $p_0$.
\end{enumerate}
\end{rem}

Therefore, the asymptotic state of the controlled system can be approximated by the 
stationary density~\eqref{eq:fp_sol}. Evaluating this density only requires the knowledge of the effective potential $\tilde V$.
As a result, no numerical integration of the SDDE~\eqref{eq:ode_cont}, or its rescaled approximating 
SDE~\eqref{eq:rescaled_approx_sde}, is required. In Section~\ref{sec:prob_trans}, we use the stationary
density $p_0$ to approximate the probability of transitions away from the desirable equilibrium $\vc x_a$.

\section{Quantifying the probability of transitions}\label{sec:prob_trans}
In this section, we estimate the probability of transitions away from the
desirable equilibrium $\vc x_a$ towards an undesirable (or `bad') state $\vc x_b\in\mathbb R^n$. 
Although our results hold for any state $\vc x_b$ (not necessarily an equilibrium), 
we are particularly interested in the case where $\vc x_b\neq \vc x_a$ is a local minimum of the 
potential $V$.

In order to quantify the transitions away from the desirable equilibrium $\vc x_a$,
we consider the ratio of the probability of $\mathbfcal X$ being close to
$\vc x_b$ to the probability of $\mathbfcal X$ being close to
$\vc x_a$. More precisely, we define the ratio
\begin{equation}
r_c(\epsilon,\vc x_a,\vc x_b) := \frac{\Pr(|\mathbfcal X-\vc x_b|\leq \epsilon)}{\Pr(|\mathbfcal X-\vc x_a|\leq \epsilon)},
\label{eq:rc}
\end{equation}
where the subscript $c$ denotes the ratio corresponding to the controlled system,
and $\epsilon>0$ is a small radius.
The smaller the ratio $r_c(\epsilon,\vc x_a,\vc x_b)$, the lower is the probability of
transitions away from the desirable equilibrium $\vc x_a$ and towards the 
undesirable equilibrium $\vc x_b$.

If $\mathbfcal X$ is a solution of the controlled SDDE~\eqref{eq:ode_cont}, then
the probabilities in Eq.~\eqref{eq:rc} can be estimated from an ensemble of long term 
numerical simulations. However, if we work with the rescaled 
approximating SDE~\eqref{eq:rescaled_approx_sde}, 
such expensive numerical simulations are unnecessary.
Note that the relevant regime here is the asymptotic steady state of the system
after the initial transients have decayed. Therefore, the probabilities in Eq.~\eqref{eq:rc}
are taken with respect to the steady state probability density~\eqref{eq:fp_sol}.

The following theorem 
establishes an upper bound for the ratio $r_c$
if $\mathbfcal X$ is given by the rescaled approximating SDE~\eqref{eq:rescaled_approx_sde}.
Here, $\mathcal B_\epsilon(\vc x_0)=\{\vc x\in\mathbb R^n: |\vc x-\vc x_0|\leq \epsilon\}$
denotes the closed ball of radius $\epsilon$ centered at $\vc x_0$.

\begin{thm}\label{thm:r_c}
	Assume (A1) and (A2), and let $\epsilon>0$ be small enough such that $\mathcal B_\epsilon(\vc x_a)\cap \mathcal B_\epsilon(\vc x_b)=\emptyset$.
	If the stochastic process $\mathbfcal X(s)$ is generated by the rescaled approximating SDE~\eqref{eq:rescaled_approx_sde},
	then 
	\begin{align}
	r_c(\epsilon,\vc x_a,\vc x_b) \leq 
	\exp\Big[ & \frac{\delta V(\vc x_a,\vc x_b,\epsilon)}{\tilde \nu}+\frac{\epsilon^2}{\beta^2} \nonumber\\
	&-\frac{\min_{\vc x\in\mathcal B_\epsilon(\vc x_b)}|\vc x-\vc x_a|^2}{\beta^2}\Big],
	\label{eq:rc_ub}
	\end{align}
	where 
	\begin{subequations}
		\begin{equation}
		\delta V(\epsilon,\vc x_a,\vc x_b) := 
		\max_{\vc x\in\mathcal B_\epsilon(\vc x_a)} V(\vc x)-\min_{\vc x\in\mathcal B_\epsilon(\vc x_b)} V(\vc x),
		\end{equation}
		\begin{equation}
		\beta = \frac{\sigma}{\sqrt{\alpha(1-\alpha\tau)}}, \quad \tilde \nu = \frac{\tilde \sigma^2}{2} = \frac{\sigma^2}{2(1-\alpha\tau)}.
		\end{equation}
	\end{subequations}
\end{thm}

\begin{proof}
See Appendix~\ref{app:r_c}.
\end{proof}

Figure~\ref{fig:schem_rc} depicts the quantities appearing in 
Theorem~\ref{thm:r_c}. Smaller upper bounds in Eq.~\eqref{eq:rc_ub} guarantee lower
probability of transition from the desirable equilibrium $\vc x_a$ to the undesirable equilibrium $\vc x_b$.
There are two primary competing factors contributing to this upper bound. 
First, a larger distance between the desirable equilibrium $\vc x_a$ and
the undesirable equilibrium $\vc x_b$ leads to a smaller upper bound. 
This distance is approximated by $\min_{\vc x\in\mathcal B_\epsilon(\vc x_b)}|\vc x-\vc x_a|^2$, 
and the upper bound vanishes exponentially fast as this quantity grows. 

\begin{figure}
\centering
\includegraphics[width=.45\textwidth]{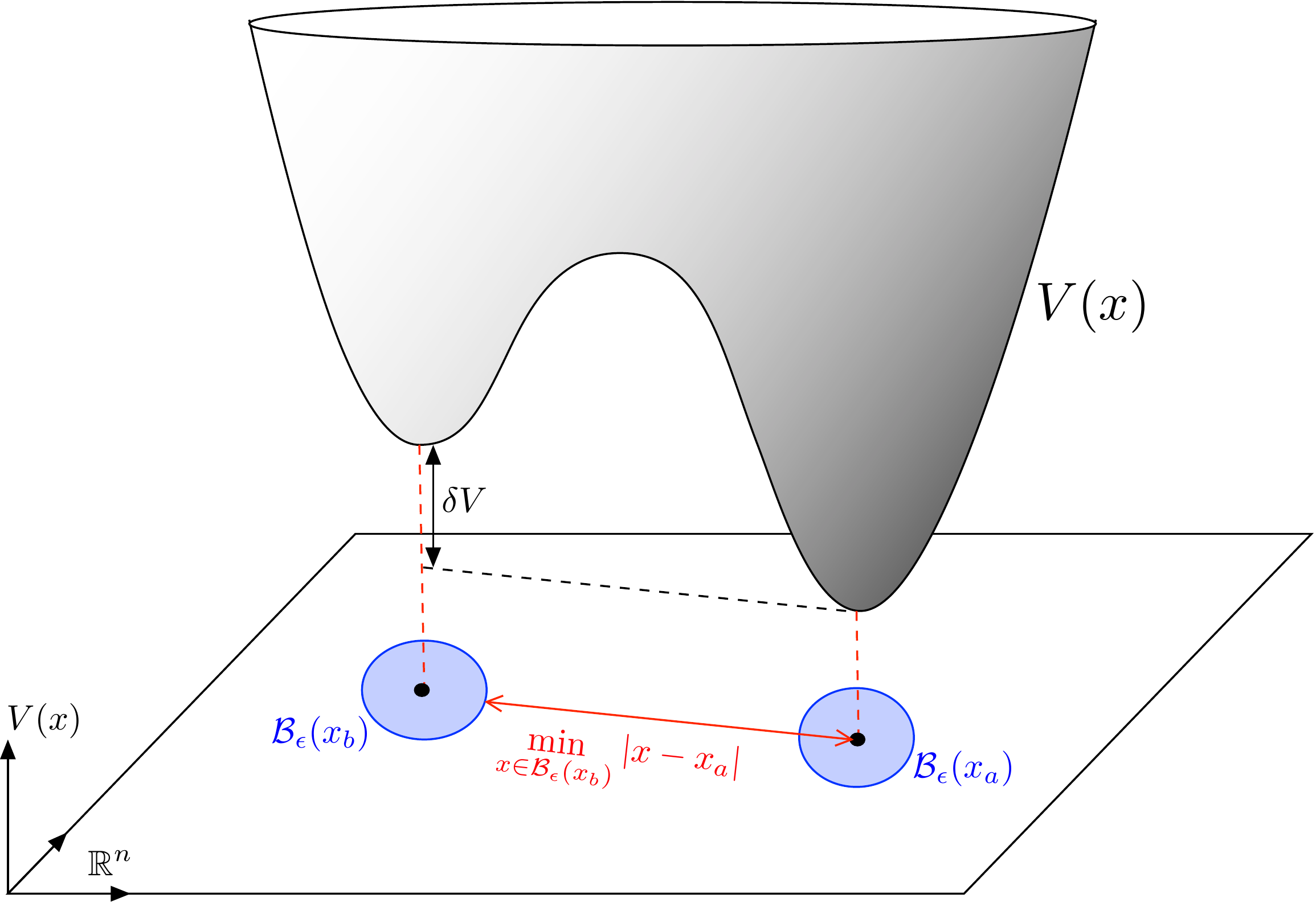}
\caption{A schematic showing various quantities appearing in Theorem~\ref{thm:r_c}.
The balls $\mathcal B_\epsilon(\vc x_a)$ and $\mathcal B_\epsilon(\vc x_b)$ denote
the $\epsilon$ neighborhoods of the desirable equilibrium $\vc x_a$ and the undesirable equilibrium 
$\vc x_b$, respectively. The surface is the graph of the potential $V$, and 
$\delta V$ denotes the potential difference at the equilibrium points as $\epsilon\to 0^+$, i.e., 
$\lim_{\epsilon\to 0^+}\delta V(\epsilon,\vc x_a,\vc x_b)$.}
\label{fig:schem_rc}
\end{figure}

A second contributing factor is the potential gap $\delta V(\epsilon,\vc x_a,\vc x_b)$.
If $\max_{\vc x\in\mathcal B_\epsilon(\vc x_a)} V(\vc x)<\min_{\vc x\in\mathcal B_\epsilon(\vc x_b)}V(\vc x)$, 
then the potential gap also contributes to
reducing the upper bound, hence reducing the probability of transitions. 
On the other hand, if $\max_{\vc x\in\mathcal B_\epsilon(\vc x_a)} V(\vc x)>\min_{\vc x\in\mathcal B_\epsilon(\vc x_b)}V(\vc x)$, then the potential gap increases the upper bound. 
In this case, the upper bound on the probability of transition depends
on the competition between the distance between two equilibria $\vc x_a$ and $\vc x_b$,
and the potential gap $\delta V$.

The following corollary shows more clearly the contributions of these two competing factors
by taking the limit $\epsilon\to 0$.

\begin{cor}
Assume (A1) and (A2). If the stochastic process $\mathbfcal X(s)$ is generated by the rescaled approximating SDE~\eqref{eq:rescaled_approx_sde},
then 
\begin{align}
\lim_{\epsilon\to 0^+} r_c(\epsilon;\vc x_a,\vc x_b) \leq 
\exp\Big\{ -\frac{\alpha(1-\alpha\tau)}{\sigma^2}\big[ & \frac{2}{\alpha}
\left(V(\vc x_b)-V(\vc x_a)\right)\nonumber\\ 
&+ |\vc x_b-\vc x_a|^2
\big]\Big\}.
\label{eq:rc_ub_0}
\end{align}
\label{cor:rc}
\end{cor}
\begin{proof}
This follows from Theorem~\ref{thm:r_c} by taking the limit $\epsilon\to 0^+$ in the 
inequality~\eqref{eq:rc_ub}, and noting that $\tilde \nu = \alpha\beta^2/2$.
\end{proof}

Therefore, in the limit $\epsilon\to 0^+$ and prescribed delay $\tau$ and gain $\alpha$, 
the upper bound is defined by the potential gap $V(\vc x_b)-V(\vc x_a)$ and 
the distance $|\vc x_b-\vc x_a|$. A larger distance between the equilibria always 
leads to a lower probability of transition. If $V(\vc x_b)>V(\vc x_a)$, the
potential gap also contributes to reduce the probability of transitions. However, 
if $V(\vc x_b)<V(\vc x_a)$, then the upper bound depends on the competition between 
the potential difference and the distance between the equilibria. 
In particular, if $V(\vc x_a)-V(\vc x_b)  \gg |\vc x_b-\vc x_a|^2$, then
the control may not lead to a meaningful reduction in the transition probability.

Note that in a given problem, the potential gap $V(\vc x_b)-V(\vc x_a)$
and the distance $|\vc x_b-\vc x_a|$ are prescribed. Therefore, the only free parameters are 
the delay $\tau$ and the control gain $\alpha$. 
Assuming no delay, i.e. $\tau=0$, the gain $\alpha$ can be arbitrary large 
and eventually all transitions are mitigated (since the upper bound vanishes as $\alpha\to \infty$).
However, in the presence of delay $\tau>0$, the gain $\alpha$ cannot be arbitrarily large since
$\alpha\tau<1$. In this case, the gain $\alpha$ needs to be chosen judiciously.

We emphasize that the upper bounds in Theorem~\ref{thm:r_c} and Corollary~\ref{cor:rc}
are valid for the stochastic process $\mathbfcal X(s)$ generated by the 
rescaled approximating SDE~\eqref{eq:rescaled_approx_sde},
and not the original controlled SDDE~\eqref{eq:ode_cont}. 
Nonetheless, we show in the next session, that for
short time delays $\tau$, these upper bounds are reliable approximations 
for the controlled system~\eqref{eq:ode_cont}. 

\section{Numerical results}\label{sec:numerics}
The purpose of this section is twofold. First, we demonstrate with 
two examples that the controlled system~\eqref{eq:ode_cont} does indeed
mitigate the transitions away from the desirable equilibrium.
Second, we demonstrate the validity of the rescaled approximating SDE~\eqref{eq:rescaled_approx_sde} as an approximation of the controlled system
when the delay $\tau$ is relatively small.
Recall that mathematical analysis of the SDDE~\eqref{eq:ode_cont} is 
cumbersome and its numerical simulations are expensive. 
In contrast, there are many useful analytical results for the rescaled approximating SDE
as presented in sections~\ref{sec:approxSDE},~\ref{sec:p0} and \ref{sec:prob_trans}.

In a given problem, the potential $V$ and the desirable equilibrium $\vc x_a$ are prescribed. 
Therefore, the only free parameters are the delay $\tau$ and the control gain $\alpha$.
The delay $\tau$ is a consequence of the fact that the 
measurement of the state $\mathbfcal X$,
and the consequent actuation of the control, introduce a delay in practice.
In the following examples, we assume that this delay is made as small as possible so that
the approximation~\eqref{eq:taylor} is valid. 

For comparison, Monte Carlo simulations of the uncontrolled SDE~\eqref{eq:ode}
and the controlled SDDE~\eqref{eq:ode_cont} are carried out.  
This requires generating sample trajectories of the stochastic equations by numerical integration.
These simulations are carried out by the
predictor-corrector scheme of Cao et al.~\cite{cao2015}. 
For completeness, we review this numerical scheme in Appendix~\ref{app:num_sdde}.

\subsection{One-dimensional potential}\label{sec:DW_1d}
\begin{figure}
	\centering
	\includegraphics[width=.35\textwidth]{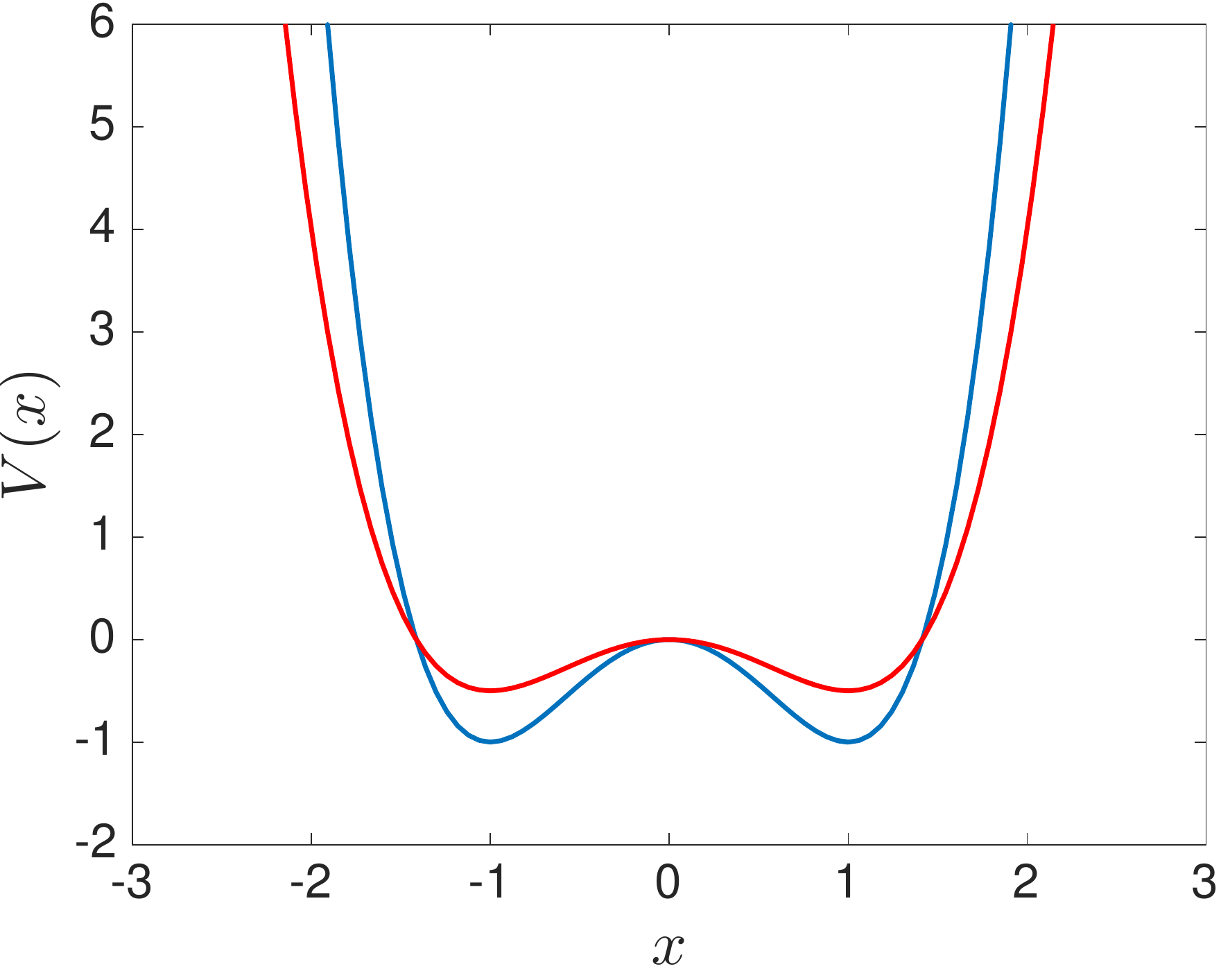}
	\caption{The one-dimensional potential function $V(x)=Ax^2+Bx^4$. We set $B=c^2$ and $A=-2B$.
		The potential for two parameter values $c=1$ (blue) and $c=1/\sqrt 2$ (red) is shown.
	}
	\label{fig:DW_V}
\end{figure}

We first consider a one-dimensional example with the potential 
\begin{equation}
V(x) = Ax^2+Bx^4,
\end{equation}
where $A<0$ and $B>0$. We assume $B=c^2$, for some constant $c$, and 
$A=-2B$. The depth of the potential well is controlled by the parameter $c=A/2\sqrt{B}$
(see, e.g., Ref.~\onlinecite{blaise2012}).
Figure~\ref{fig:DW_V} shows the potential $V$ for two choices of the parameter $c$.
\begin{figure*}
	\centering
	\includegraphics[width=.8\textwidth]{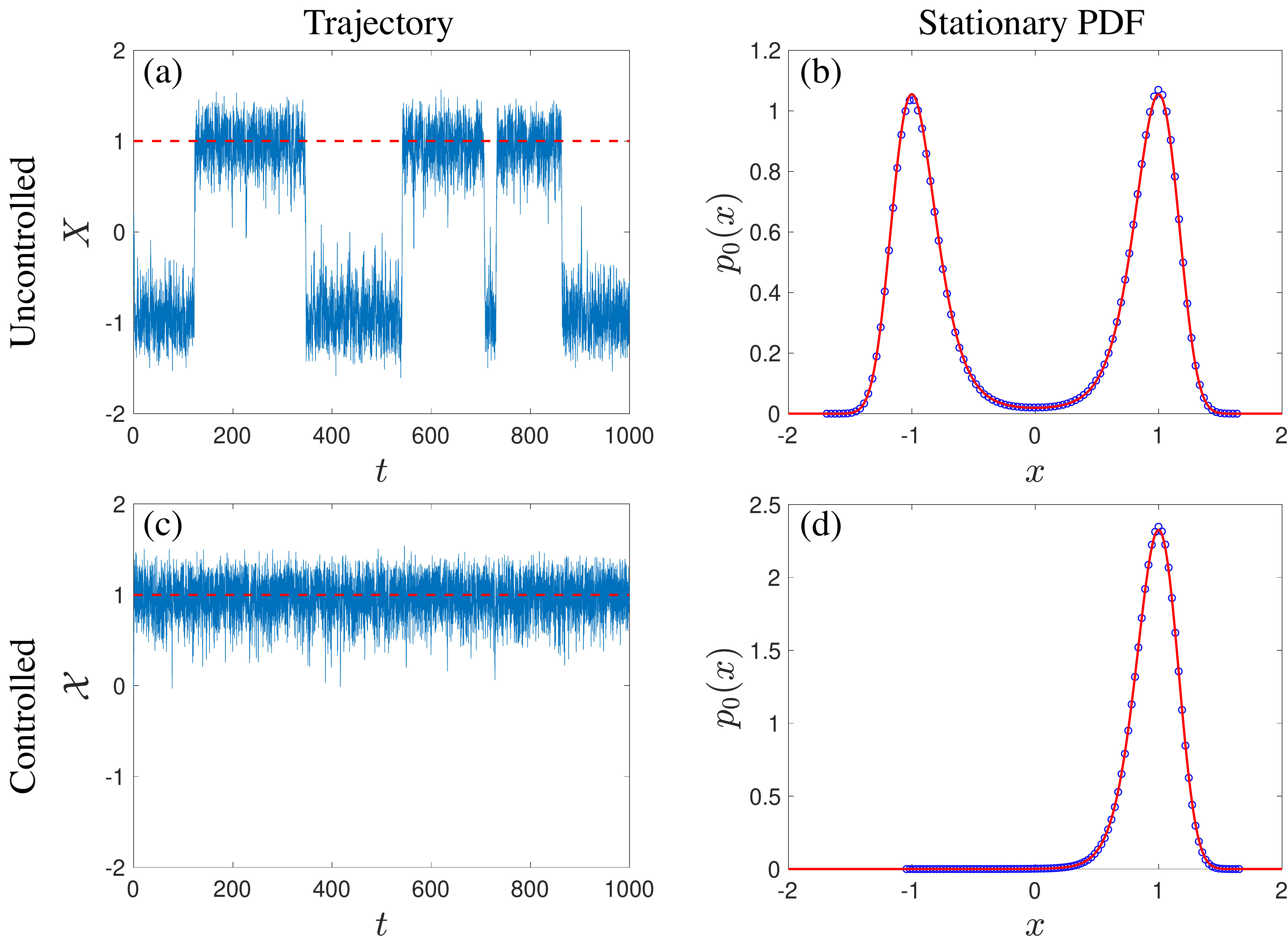}
	\caption{
		(a) A sample trajectory of the uncontrolled system~\eqref{eq:DW_1d_uncont}. The red dashed line marks the desirable equilibrium $x_a=+1$. 
		(b) Stationary PDF of the uncontrolled system. The blue circles mark the PDF estimated from longterm direct numerical simulations of the uncontrolled system~\eqref{eq:DW_1d_uncont}. 
		The solid red curve is the exact stationary PDF (see Theorem~\ref{thm:EU}).
		(c) A sample trajectory of the controlled system~\eqref{eq:DW_1d} with $\tau=0.1$ and
		$\alpha=1$. (d) The stationary PDF of the controlled system. The blue circles mark the PDF estimated from longterm direct numerical simulations of the controlled system~\eqref{eq:DW_1d}. 
		The solid red curve is the exact stationary PDF of the rescaled approximating SDE (see Theorem~\ref{thm:EU}).
	}
	\label{fig:DW_FP_1d}
\end{figure*}

The corresponding deterministic differential
equation $\dot x = -V'(x)$ has three fixed points located at $x=0,\pm \sqrt{-A/2B}$. The origin is an unstable fixed point while 
$\pm\sqrt{-A/2B}$ are stable. If the initial condition $x(0)=x^0$ satisfies $x^0>0$, then 
the trajectories of the deterministic system converge asymptotically to the stable fixed point $x=+\sqrt{-A/2B}$. Conversely, if $x^0<0$, the trajectories converge to the stable fixed point $x=-\sqrt{-A/2B}$.
However, addition of noise allows for transitions between the two fixed points
without the trajectory ever settling onto one of them.

In the following, we set $c=1/\sqrt{2}$, $B=c^2$ and $A=-2B$. This implies that the fixed points are located at $x=0,\pm 1$. The resulting uncontrolled SDE reads
\begin{equation}
\id X = -( -2X+2X^3)\id t + \sigma \id W,
\label{eq:DW_1d_uncont}
\end{equation}
where the noise intensity is $\sigma =1/2$.
Note that $V'(x) = -2x+2x^3$.
Figure~\ref{fig:DW_FP_1d}(a) shows a sample trajectory of this SDE which switches
randomly between the fixed points $x=\pm 1$.
Figure~\ref{fig:DW_FP_1d}(b) shows the resulting stationary distribution $p_0$
(blue circles), estimated from $200$ direct numerical simulations.
Each simulation is $1000$ time units long with the sampling rate of $0.01$, 
resulting in $2\times 10^7$ samples.
Figure~\ref{fig:DW_FP_1d}(b) also shows the exact stationary density~\eqref{eq:fp_sol}
with $\alpha=0$ (solid red line) which is in agreement with direct numerical simulations.

Next we consider the controlled system. Without loss of generality, we assume that
$x_a=+1$ is the desirable equilibrium
and $x_b=-1$ is the undesirable equilibrium. In this case,
the controlled SDDE reads
\begin{equation}
\id \mathcal X = -\left( -2\mathcal  X + 2\mathcal  X^3\right)\id t 
-\alpha (\mathcal X(t-\tau)-1)\id t + \sigma \id W, 
\label{eq:DW_1d}
\end{equation}
where the term $-\alpha (\mathcal X(t-\tau)-1)$ is the time-delay feedback control. 
For $\alpha=0$ we recover the uncontrolled system~\eqref{eq:DW_1d_uncont}.
Figure~\ref{fig:DW_FP_1d}(c) shows a sample trajectory of the controlled system
with $\alpha=1$ and $\tau=0.1$. Note that this trajectory has
no transitions to the neighborhood of the undesirable equilibrium. 
Ensemble simulations exhibit a similar behavior. Figure~\ref{fig:DW_FP_1d}(d)
shows that probability density of $\mathcal X(t)$ from $200$ simulations, 
each consisting of $1000$ time units. In contrast to the uncontrolled system,
the stationary density of the controlled system is unimodal with most of the
density concentrated around the desirable equilibrium $x_a=+1$. 
As a result, transitions to the neighborhood of the undesirable equilibrium $x_b=-1$
are very unlikely. 

Recall that in this experiment, we chose the control gain $\alpha=1$ and
the delay $\tau=0.1$. How would one choose the values of these variables?
The answer is not clear if we attempt to work with the controlled SDDE~\eqref{eq:ode_cont}. 
However, if the delay is small, the results found in sections~\ref{sec:p0} and~\ref{sec:prob_trans}
provide an inexpensive method to choose these parameters.

The delay must be chosen as small as possible. In practice, there is a lower bound
on how small the delay may be, dictated by the amount of time it takes to measure
the state $\mathcal X$. Lets fix the delay at $\tau=0.1$. 
Figure~\ref{fig:effV_1d} shows the effective potential $\tilde V$ as $\alpha$ varies.
For $\alpha=0$, we recover the uncontrolled potential $V$ with two minima at 
$x_a=+1$ and $x_b=-1$. As $\alpha$ increases, the effective potential around
the undesirable equilibrium $x_b=-1$ increases while the effective potential at the desirable 
equilibrium $x_a=+1$ remains unchanged. Eventually, for $\alpha$ large enough, the minimum
around the undesirable equilibrium disappears and the effective potential has a unique minimum
at the desirable equilibrium $x_a=+1$ (see, e.g., the curve corresponding to $\alpha=1$).

Correspondingly, at $\alpha=0$ (no control), the stationary density $p_0$ is
bimodal with equal probability of visiting each of the equilibria $x_a=+1$ and $x_b=-1$.
However, as $\alpha$ increases the density around the undesirable equilibrium $x_b=-1$
decreases. Eventually, for large enough $\alpha$ (e.g., $\alpha=1$), 
the stationary density becomes unimodal concentrated around the desirable
equilibrium $x_a=+1$. 

These observations reveal that $\alpha=1$ is a good choice for the control gain
at delay $\tau=0.1$. Note that since the effective potential $\tilde V$ and 
the stationary density $p_0$ are known analytically, this analysis 
is computationally inexpensive. In particular, it does not require 
the numerical integration of any SDEs or SDDEs. 

Furthermore, for $\tau=0.1$ and $\alpha=1$, the upper bound in 
Corollary~\ref{cor:rc} is approximately $5.6\times 10^{-7}$. 
This roughly means that the probability of the state being around 
the undesirable equilibrium $x_b=-1$ is $5.6\times 10^{-7}$
times smaller than the state being around the desirable equilibrium $x_a=+1$.

There is no general recipe for choosing the delay $\tau$.
As mentioned earlier, it must be as small as possible, so that the 
approximating SDE is a reasonable approximation of the controlled SDDE. 
It is still interesting to investigate whether the controlled system
would mitigate transitions away from the desirable equilibrium 
for larger delays. Figure~\ref{fig:large_delay} shows how the stationary density
varies as the delay increases. For $\tau=0.2$, the stationary density is unimodal, 
centered at the desirable equilibrium with negligible probability of transitioning to 
the undesirable equilibrium. In addition, the exact stationary density of the approximating
SDE is still a good approximation to the numerically estimated PDF. 

The results are similar for $\tau=0.4$. However, the exact PDF of the approximating SDE starts to deviate 
from the numerically estimated PDF of the SDDE. This is to be expected since, 
as the delay increases, the neglected $\mathcal O(\tau^2)$ terms become significant.
Even at $\tau=2$, the stationary PDF is unimodal and the probability of transitions is negligible. 
Note that at $\tau=2$, the approximating SDE is no longer valid since $\alpha\tau>1$.

If we keep increasing the delay, eventually for $\tau>5$ the control fails in the sense that the
stationary PDF becomes bimodal again with non-negligible density around the undesirable equilibrium $x_b=-1$.

\begin{figure}
\centering
\includegraphics[width=.35\textwidth]{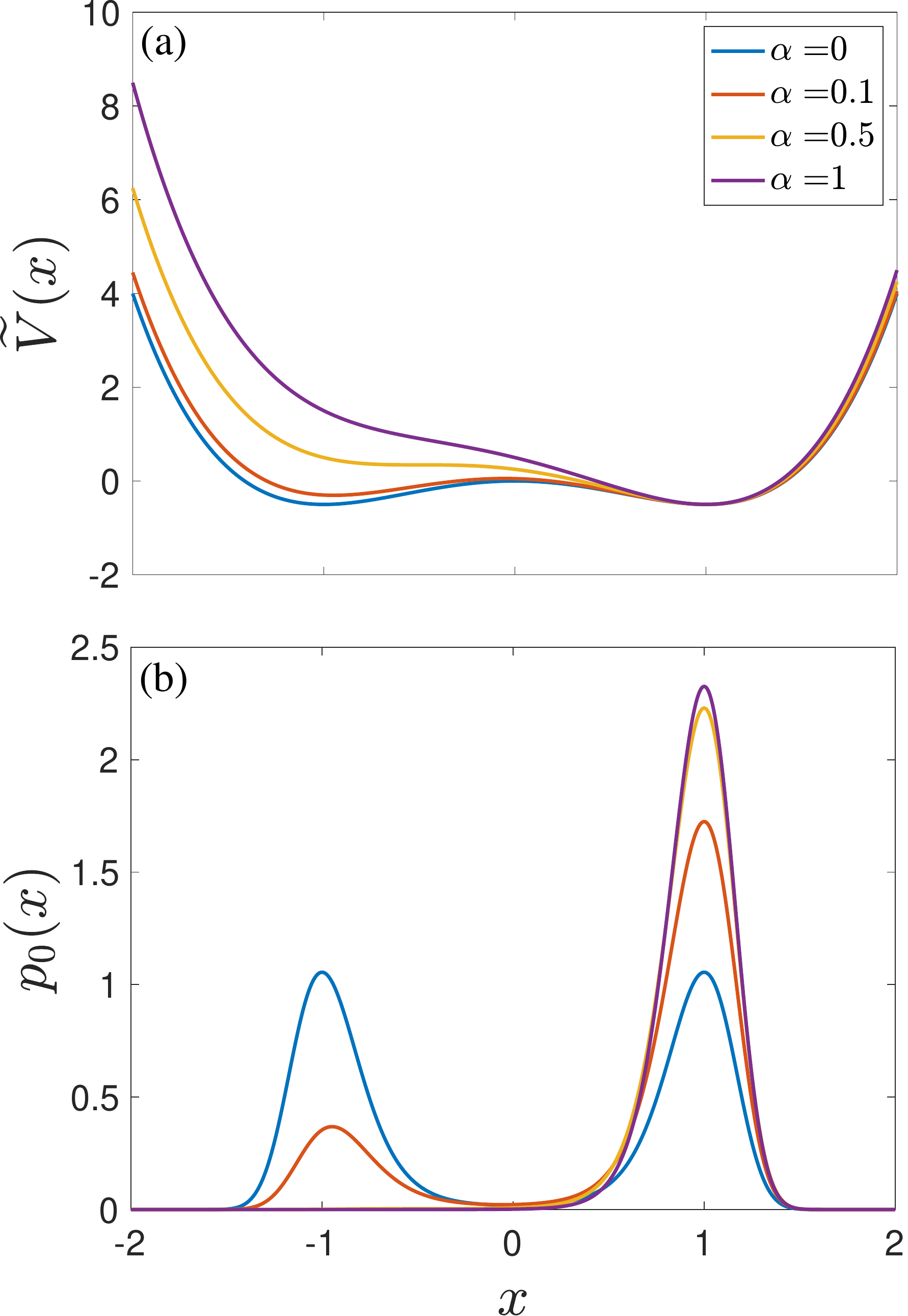}
\caption{The effective potential and the stationary density
for various control gains $\alpha$.
(a) The effective potential $\tilde V(x)$. The case with $\alpha=0$ corresponds to the uncontrolled system.
(b) The corresponding stationary densities $p_0(x)$.
}
\label{fig:effV_1d}
\end{figure}

\begin{figure*}
	\centering
	\includegraphics[width=.9\textwidth]{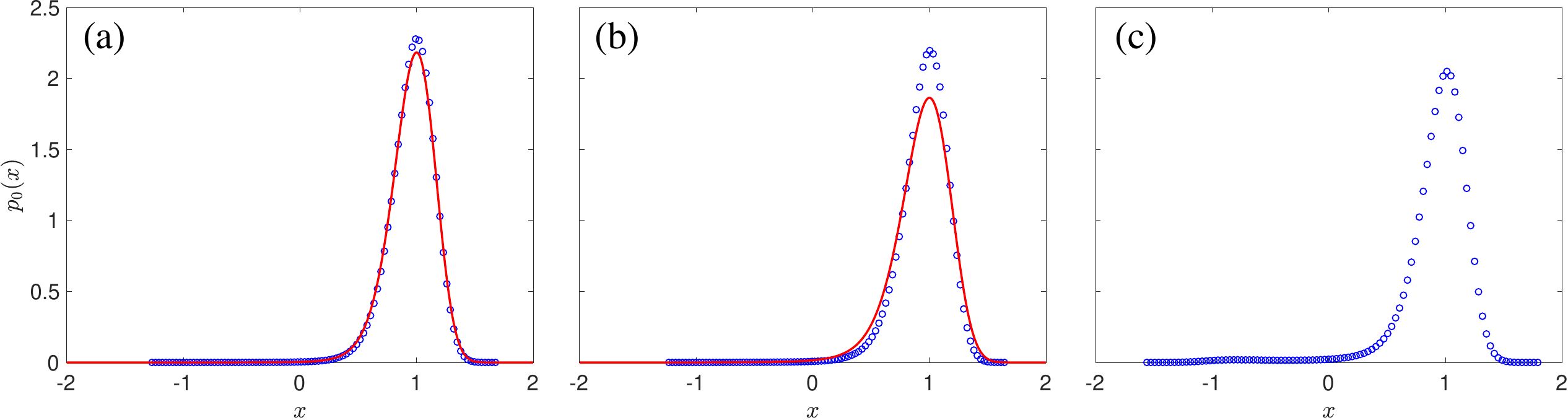}
	\caption{The stationary density of the controlled system with $\alpha=1$. 
		The blue circles mark the PDF estimated from direct numerical simulations of the SDDE~\eqref{eq:ode_cont}.
		The red solid lines mark the exact stationary density of the rescaled approximating SDE~\eqref{eq:rescaled_approx_sde}. (a) $\tau=0.2$, (b) $\tau=0.4$ and (c) $\tau=2$.
		Note that, for $\tau=2$, the rescaled approximating SDE is not valid since $\alpha\tau>1$.
	}
\label{fig:large_delay}
\end{figure*}

\subsection{Two-dimensional potential}\label{sec:DW_2d}
In this section, we present a two-dimensional example
with the potential
\begin{equation}
V(\vc x)=a_0 x^4-a_1 x^2 +a_2 y^2 -a_3 x^2y-a_4 \left(x-x_0\right)^2
\label{eq:VueDuffing}
\end{equation}
where $\vc x=(x,y)$. The parameters are 
$a_0=1$, $a_1=3/8$, $a_2=1/4$, $a_3=\sqrt{4a_2(a_0-2a_1)}=1/2$, $a_4=0.02$
and $x_0=\sqrt{2a_1a_2/(4a_0a_2-a_3^2)}=1/2$.
This potential has two minima at 
\begin{equation}
\vc x_a = (x_0,a_3 x_0^2/2a_2), \quad \vc x_b = (-0.5254, 0.2760),
\end{equation}
and a saddle at $\vc x_s=(0.0254, 0.0006)$ (see figure~\ref{fig:ueDuffing_V}).

\begin{figure*}
\centering
\includegraphics[width=.9\textwidth]{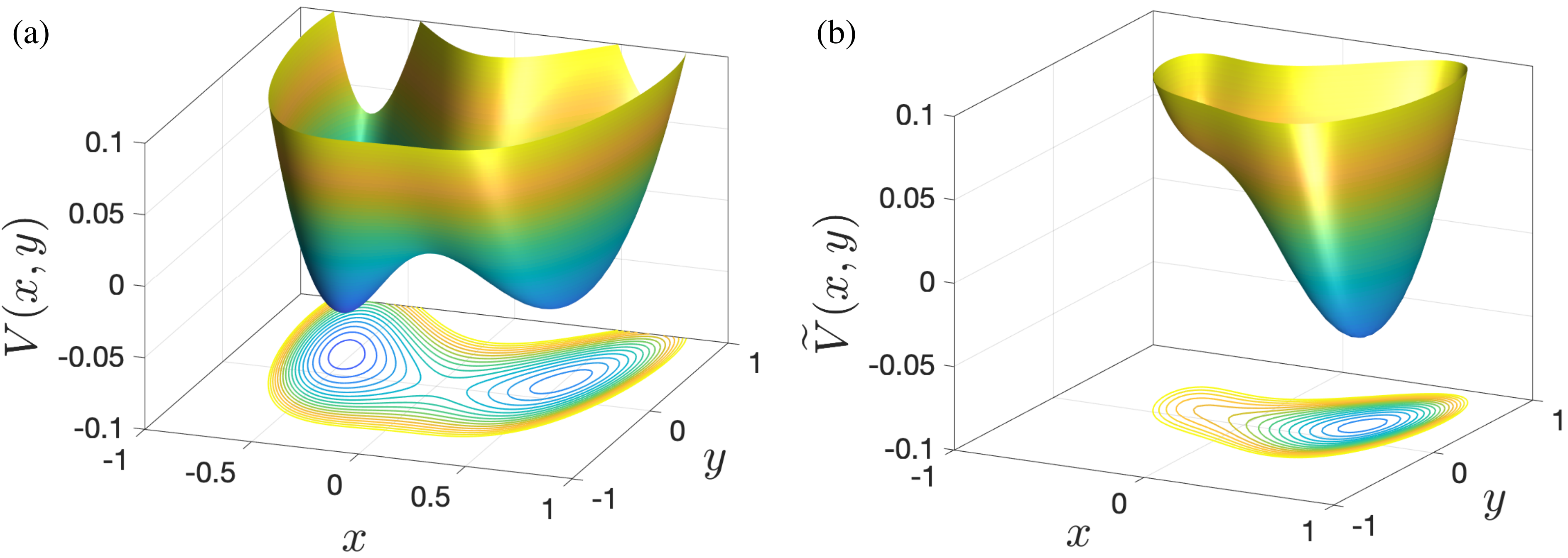}
\caption{The two-dimensional potential function.
(a) The uncontrolled potential~\eqref{eq:VueDuffing}. 
(b) The effective potential $\tilde V(\vc x)=V(\vc x)+\alpha\|\vc x-\vc x_a\|^2/2$ corresponding to the controlled system with $\tau=0.1$ and $\alpha=0.3$.}
\label{fig:ueDuffing_V}
\end{figure*}

Again the minima $\vc x_a$ and $\vc x_b$ are both stable fixed points of the deterministic
system $\dot{\vc x}=-\bnabla V(\vc x)$. Addition of noise allows for rare transitions between the 
two equilibria. Our goal is to mitigate the transitions away from the desirable equilibrium $\vc x_a$.

\begin{figure*}
\centering
\includegraphics[width=.8\textwidth]{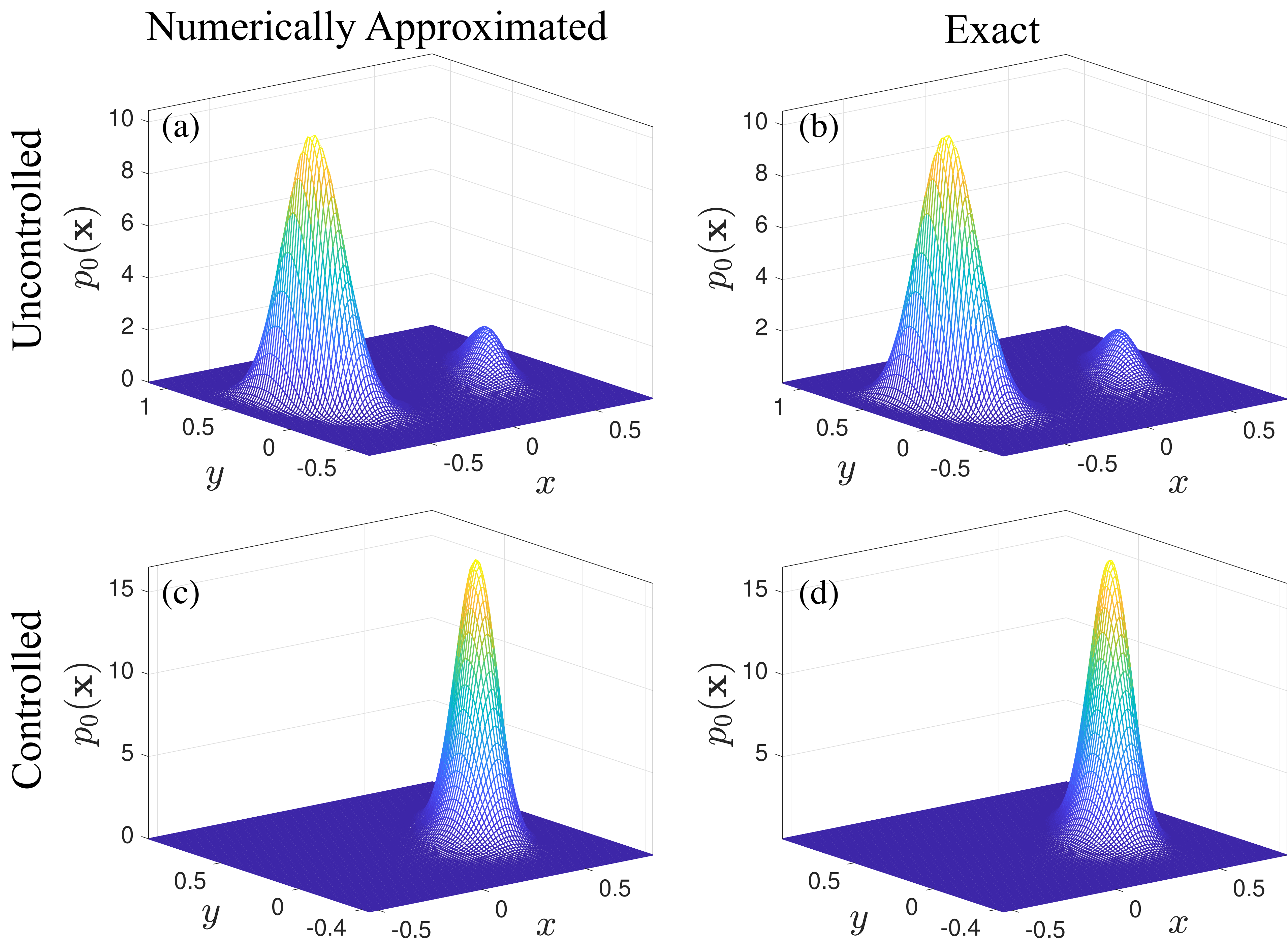}
\caption{Stationary PDF corresponding to the two-dimensional potential~\eqref{eq:VueDuffing}. 
(a) Stationary PDF of the uncontrolled system approximated from direct numerical
simulations. 
(b) Exact stationary PDF of the uncontrolled system. 
(c) Stationary PDF of the controlled SDDE~\eqref{eq:ode_cont} approximated from direct numerical
simulations. 
(d) Exact PDF of the rescaled approximating SDE~\eqref{eq:rescaled_approx_sde}.
In panels (c) and (d), we have $\tau=0.1$ and $\alpha=0.3$ for the delay and control gain, respectively.
}
\label{fig:ueDuffing_pdf}
\end{figure*}

The function $V$ is designed so that the potential is deeper at the minimum $\vc x_b$
(see figure~\ref{fig:ueDuffing_V}(a)).
As a result, the trajectories of the uncontrolled SDE~\eqref{eq:ode}
are more likely to be found around this undesirable equilibrium. 
This is confirmed by direct numerical simulations of the uncontrolled SDE~\eqref{eq:ode} 
and estimating the stationary equilibrium as shown in figure~\ref{fig:ueDuffing_pdf}(a).
For reference, the exact stationary PDF of the uncontrolled system is
shown in figure~\ref{fig:ueDuffing_pdf}(b), which is in agreement with the
numerically estimated PDF. In both cases the noise intensity is $\sigma=0.15$.

Next we consider the controlled system with delay $\tau=0.1$. As before, the delay is chosen to be small
so that the approximating SDE is a reasonable approximation of the controlled SDDE. 
Gradually increasing the control gain $\alpha$, we find that $\alpha=0.3$ is adequate for 
mitigating the transitions away from the desirable equilibrium $\vc x_a$. 
This is shown in figure~\ref{fig:ueDuffing_V}(b), which depicts the
effective potential $\tilde V(\vc x)=V(\vc x)+\alpha|\vc x-\vc x_a|^2/2$  with $\alpha=0.3$.
This effective potential has a unique minimum at the point $\vc x_a$.

Furthermore, for $\tau=0.1$ and $\alpha=0.3$, the upper bound in Corollary~\ref{cor:rc}
is approximately $1.6\times 10^{-6}$, indicating low probability of transitions towards the undesirable
equilibrium. 
We emphasize again that determining the control gain $\alpha$ can be done merely by
analyzing the function $\tilde V$, the stationary density $p_0$ and the upper bound~\eqref{eq:rc_ub_0}, 
without any numerical simulations of the controlled SDDE 
or its approximating SDE. 

Figure~\ref{fig:ueDuffing_pdf}(c-d) compares the numerically estimated 
stationary PDF of the controlled system (SDDE) and the exact stationary 
PDF of its approximating SDE. The two PDF are in good agreement. Furthermore, 
the PDFs are unimodal with almost all the density concentrated around the 
desirable equilibrium $\vc x_a$.  As a result, transition to the undesirable 
equilibrium $\vc x_b$ is very unlikely, and the mitigation has been successful. 

\section{Conclusions}\label{sec:concl}
We have demonstrated the effectiveness of delay feedback control 
for mitigation of transitions in stochastic multistable systems driven by the gradient of a potential.  
In the absence of control, small amounts of noise can lead to rare transitions
between the stable equilibria of the system. The delay feedback control
mitigates these transitions by stabilizing an arbitrary equilibrium. 

We added a delay feedback control term that turns the original stochastic differential equation (SDE) into a
stochastic delay differential equation (SDDE). 
For small delay $\tau$, we showed that this SDDE can be approximated by an SDE. 
This approximation reveals two interesting features of the control: 
(i) the delay feedback modifies the potential $V$ resulting
in an effective potential $\tilde V$ which is deeper at the location of the desirable
equilibrium $\vc x_a$. As a result, transitions away from this equilibrium become harder.
(ii) At the same time, the delay feedback intensifies the effect of noise. In particular, the 
noise intensity of the uncontrolled system $\sigma$ becomes $\sigma/\sqrt{1-\alpha\tau}$
in the controlled system, where $\alpha$ is the control gain.

Therefore, the effect of delay feedback is twofold. A stabilizing effect through deepening of the 
potential well and a destabilizing effect due to intensified noise. As a result, the delay $\tau$
and the gain $\alpha$ should be chosen appropriately in order to ensure that the 
stabilizing effect dominates the intensified noise. 

We also derived exact results for the approximating SDE that elucidate the appropriate choice of the
control delay and gain.
For a given small delay $\tau$, an appropriate 
control gain $\alpha$ can be determined by analysis of the effective potential $\tilde V$
and the stationary density $p_0$.
Since these two quantities are known analytically, the control gain can be determined 
at a low computational cost, without any need for generating sample trajectories of an
SDE (or SDDE). 
Using these results, we also derived an upper bound for the 
probability of transitions away from the desirable equilibrium
and towards an undesirable one. 
This upper bound (cf. Theorem~\ref{thm:r_c}) depends on 
the control gain $\alpha$, the delay $\tau$, the potential difference between the two equilibria
(before applying the control) and the Euclidean distance between them.

Our work can be extended in several directions. The approximating SDE 
(and its rescaled version; Theorem~\ref{thm:RASDE}) is valid for multiplicative noise, where
the diffusion tensor $\bsigma(\vc x)$ depends on space. However, our analytical results 
concerning the stationary PDF and the upper bound for transition probabilities are
limited to the case where $\sigma$ is constant. It would be attractive 
to generalize these results to the case of multiplicative noise. 

More importantly, it is desirable to design low-dimensional controllers of the form
\begin{equation}
\vc u(t-\tau)=P\left(\vc X(t-\tau)-\vc x_a\right),
\end{equation}
where $P$ is a projection onto a lower dimensional subspace.
Such controllers would be more practical for mitigating transitions in 
high-dimensional systems, since they only act on a few degrees of freedom. 
Our preliminary numerical observations (not presented here) indicate that, if
the projection is onto the least stable subspace of the fixed point $\vc x_a$, 
then the mitigation is successful. However, our analytical results do not 
immediately extend 
to this low-dimensional control. The difficulty arises from the fact 
that $P\bnabla V(\vc x)$ or $\bnabla V(P\vc x)$ cannot necessarily 
be written as gradient functions. In other words, an effective potential $\tilde V$
does not necessarily exist such that $\bnabla \tilde V(\vc x)=P\bnabla V(\vc x)$
or $\bnabla \tilde V(\vc x)=\bnabla V(P\vc x)$. A counter example is provided 
by the function $V(x,y)=xy$ and $P$ being the orthogonal projection onto the $x$ coordinate.

Finally, analytical results for the general SDE~\eqref{eq:ode_general}, 
where $\vc f(\vc X)$ is not the gradient of a potential,
would be welcome. At the moment, very little is known about the 
delay feedback control of such general stochastic systems. 
These open questions will be explored in forthcoming work. 

\section*{Acknowledgments}
I thank Gabor Stepan for fruitful discussions on delay differential equations. 
This work was partially supported by the ARO MURI Grant W911NF-17-1-0306 at MIT, while the author held a postdoc position in Prof. Sapsis' group.

\appendix
\section{Proof of Theorem~\ref{thm:RASDE}}\label{app:RASDE}
Consider $t$ as a function of the rescaled time $s$ so that $t(s) = (1-\alpha\tau) s$. Furthermore, 
we define a new stochastic process $\tilde{\mathbfcal X}(s)$ such that $\tilde{\mathbfcal X} (s) := \mathbfcal X(t(s))$. 
Noting that $\id t = (1-\alpha\tau) \id s$, Eq.~\eqref{eq:approx_sde} can be written as
\begin{align}
(1-\alpha\tau )\id{\tilde{\mathbfcal X}}(s) = & -(1-\alpha\tau)\bnabla V(\tilde{\mathbfcal X}) \id s \nonumber\\
& - (1-\alpha\tau)\alpha\left(\tilde{\mathbfcal X}-\vc x_a\right)\id s \nonumber\\
& +\bsigma(\tilde{\mathbfcal X})\id{\vc W}(t). 
\label{eq:RASDE_2}
\end{align}
Note that the Wiener process $\vc W(t)$ is in the original time $t$
and needs to be transformed to the rescaled time $s$. 
To this end, we define $\tilde{\vc W}(s):=\vc W(t)/\sqrt{1-\alpha\tau}$ and show that $\tilde{\vc W}(s)$ is
also a standard Wiener process. First note that $\mathbb P(\{\tilde{\vc W}(0)=\vc 0\})=\mathbb P\left(\{\vc W(0)=\vc 0\}\right)=1$.
Clearly, the increments $\tilde{\vc W}(s_2)-\tilde{\vc W}(s_1)$ are Gaussian and independent for all $s_1<s_2$, 
and $\mathbb E[\tilde{\vc W}(s_2)-\tilde{\vc W}(s_1)]= \mathbb E[\vc W(t_2)-\vc W(t_1)]/\sqrt{1-\alpha\tau}=\vc 0$. 
It remains to show that the variance of the increment is $s_2-s_1$. For simplicity, we show that for 
a scalar process $\tilde W(s)$. The generalization to the multidimensional case is straightforward.
\begin{align}
\mathbb E\left[|\tilde W(s_2)-\tilde W(s_1)|^2 \right] &= \frac{\mathbb E\left[|W(t(s_2))-W(t(s_1))|^2\right]}{1-\alpha\tau} 
\nonumber\\
&=\frac{t(s_2)-t(s_1)}{1-\alpha\tau} = s_2-s_1.
\end{align}
This proves that $\tilde{\vc W}(s)$ is a standard Wiener process. 

Substituting $\id \vc W(t) = \sqrt{1-\alpha\tau}\,\id \tilde{\vc{W}}(s)$ into Eq.~\eqref{eq:RASDE_2} and dividing
by $(1-\alpha\tau)$, we obtain
\begin{equation}
\id{\tilde{\mathbfcal X}}(s) = -\bnabla V(\tilde{\mathbfcal X}) \id s
- \alpha\left(\tilde{\mathbfcal X}-\vc x_a\right)\id s +
\frac{\bsigma(\tilde{\mathbfcal X})}{\sqrt{1-\alpha\tau}}\id\tilde{\vc W}(s). 
\end{equation}
Finally, defining $\tilde \bsigma (\vc x)=\bsigma (\vc x)/\sqrt{1-\alpha\tau}$
and omitting the tilde signs from $\tilde{\mathbfcal X}$ and $\tilde{\vc W}$, 
we obtain the desired Eq.~\eqref{eq:rescaled_approx_sde}.

\section{Proof of Theorem~\ref{thm:r_c}}\label{app:r_c}
Recall that $\Pr\left( \mathbfcal X\in\mathcal S\right) =
\int_{\mathcal S}p_0(\vc x)\id \vc x,$ for any Lebesgue-measurable set $\mathcal S$.  
Using Theorem~\ref{thm:EU}, we have
\begin{equation}
r_c = \frac{
	\int_{\mathcal B_\epsilon(\vc x_b)}e^{-V(\vc x)/\tilde\nu}e^{- \frac{\alpha}{2\tilde \nu}|\vc x-\vc x_a|^2}\id \vc x
}{
	\int_{\mathcal B_\epsilon(\vc x_a)}e^{-V(\vc x)/\tilde\nu}e^{- \frac{\alpha}{2\tilde \nu}|\vc x-\vc x_a|^2}\id \vc x	
}.
\end{equation}
Using the estimates,
$$e^{-\frac{V(\vc x)}{\tilde\nu}}\leq e^{-\frac{\min_{\vc x\in \mathcal B_\epsilon(\vc x_b)} V(\vc x)}{\tilde \nu}},\quad \forall \vc x\in \mathcal B_\epsilon(\vc x_b),$$
$$e^{-\frac{V(\vc x)}{\tilde\nu}}\geq e^{-\frac{\max_{\vc x\in \mathcal B_\epsilon(\vc x_a)} V(\vc x)}{\tilde \nu}},\quad \forall \vc x\in \mathcal B_\epsilon(\vc x_a),$$
we obtain
\begin{equation}
r_c \leq \exp\left[ \frac{\delta V(\epsilon,\vc x_a,\vc x_b)}{\tilde \nu}\right]
\frac{
	\int_{\mathcal B_\epsilon(\vc x_b)}e^{- \frac{|\vc x-\vc x_a|^2}{\beta^2}}\id \vc x
}{
	\int_{\mathcal B_\epsilon(\vc x_a)}e^{- \frac{|\vc x-\vc x_a|^2}{\beta^2}}\id \vc x	
},
\end{equation}
where we used the fact that $\beta^2 = 2\tilde\nu/\alpha$.
Next, we observe that 
$$e^{- \frac{|\vc x-\vc x_a|^2}{\beta^2}}\leq e^{- \frac{\min_{\vc x\in \mathcal B_\epsilon(\vc x_b)}|\vc x-\vc x_a|^2}{\beta^2}},
\quad \forall \vc x\in\mathcal B_\epsilon(\vc x_b)$$
$$e^{- \frac{|\vc x-\vc x_a|^2}{\beta^2}}\geq e^{- \frac{\max_{\vc x\in \mathcal B_\epsilon(\vc x_a)}|\vc x-\vc x_a|^2}{\beta^2}}
=e^{-\frac{\epsilon^2}{\beta^2}},\quad \forall \vc x\in\mathcal B_\epsilon(\vc x_a).$$
As a result, 
\begin{equation}
r_c \leq \exp\left[ \frac{\delta V(\epsilon,\vc x_a,\vc x_b)}{\tilde \nu}\right]
\frac{
	e^{- \frac{\min_{\vc x\in \mathcal B_\epsilon(\vc x_b)}|\vc x-\vc x_a|^2}{\beta^2}}\mbox{Vol}\left(\mathcal B_\epsilon(\vc x_b)\right)
}{
	e^{-\frac{\epsilon^2}{\beta^2}}\,\mbox{Vol}\left(\mathcal B_\epsilon(\vc x_b)\right)
},
\end{equation}
where $\mbox{Vol}(\mathcal B_\epsilon(\vc x_b)) = \frac{(\epsilon\sqrt{\pi})^n}{\Gamma(\frac{n}{2}+1)}$
denotes the volume of the ball $\mathcal B_\epsilon(\vc x_b)\subset \mathbb R^n$.
This completes the proof.

\section{Numerical integration of SDDEs}\label{app:num_sdde}
Numerical integration of SDDEs in Section~\ref{sec:numerics}
are carried out by the predictor-corrector scheme developed by 
Cao et al.~\cite{cao2015}. For completeness, we present their method here. 
This numerical scheme applies to SDDEs of the form
\begin{equation}
\id\vc X(t) = \vc f(\vc X(t),\vc X(t-\tau))\id t + \bsigma(\vc X(t))\id \vc W(t),
\label{eq:sdde_cao}
\end{equation}
where $\vc f:\mathbb R^n\to \mathbb R^n$, $\vc W(t)\in\mathbb R^m$ 
is a standard multidimensional Wiener process, and $\bsigma(\vc X) \in\mathbb R^{n\times m}$ 
is the diffusion matrix. For the smoothness requirements of the maps $\vc f$
and $\bsigma$ refer to Section 2 of Cao et al.~\cite{cao2015}.
In the context of our paper, we have
\begin{equation}
\vc f(\vc X(t),\vc X(t-\tau)) = -\bnabla V(\vc X(t)) -\alpha \left( \vc X(t-\tau)-\vc x_a\right). 
\end{equation}

Consider a small time step $h>0$, and let $\vc X^i$ denote the numerical approximation of $\vc X(ih)$, i.e., 
$\vc X^i \simeq \vc X(ih)$. Furthermore, we assume that the delay $\tau$ is an integer multiple of the
time step $h$, such that $\tau = j h$ for some $j\in\mathbb N\cup \{0\}$. 
We denote the elements of the vector $\vc W(t)$ by $W_{\ell}(t)$, 
and write $W_\ell^i$ for $W_\ell(ih)$.
Similarly, the columns of 
$\bsigma$ are denoted by $\bsigma_{\ell}$, 
\begin{equation}
\bsigma = [\bsigma_1|\bsigma_2|\cdots|\bsigma_m]. 
\end{equation}

In this set-up, the predictor-corrector method can be written as the two-step scheme,
\begin{subequations}
\begin{equation}
\overline{\vc X}^{i+1} = \vc X^i + h\vc f\left(\vc X^i, \vc X^{i-j}\right)
+\sum_{\ell=1}^{m} \bsigma_\ell(\vc X^i)\Delta W_\ell^i,
\end{equation}
\begin{align}
\vc X^{i+1} = & \vc X^i + \frac{h}{2}\left[\vc f\left(\vc X^i, \vc X^{i-j}\right)
+\vc f\left(\overline{\vc X}^{i+1}, \vc X^{i-j+1}\right)\right]\nonumber\\
& +\frac12\sum_{\ell=1}^{m} \left[\bsigma_\ell(\vc X^i)+\bsigma_\ell(\overline{\vc X}^{i+1})\right]\Delta W_\ell^i,
\end{align}
\label{eq:cao}
\end{subequations}
where $\Delta W_\ell^i$ denote the increments of the Wiener process, 
$\Delta W_\ell^i = W_\ell^{i+1}-W_\ell^i$. 

If the maps $\vc f$ and $\bsigma$ are smooth enough, 
the above scheme has a global mean-square convergence rate of $\mathcal O(h)$
(see Theorem 2.2 in Ref.~\onlinecite{cao2015}). The numerical integrations in Section~\ref{sec:numerics}
are carried out using the scheme~\eqref{eq:cao} with $h=0.01$.

%

\end{document}